\documentclass[twoside,a4paper,12pt,centertags]{amsart}
\usepackage{amsmath,amssymb,verbatim,vmargin, amsbsy}
\usepackage[bookmarks=true]{hyperref}   
\usepackage[all]{xy}

\theoremstyle{plain}
\newtheorem{thm}{Theorem}[section]
\newtheorem{lem}[thm]{Lemma}

\newtheorem{prop}[thm]{Proposition}

\theoremstyle{definition}

\newtheorem{ex}[thm]{Example}

\newcommand{\mv}[1]{\boldsymbol{#1}} 

\mathchardef\semic="303A
\newcommand{\R}{{\mathbf R}}
\newcommand{\C}{{\mathbf C}}
\newcommand{\Z}{{\mathbf Z}}
\newcommand{\mH}{{\mathcal H}}
\newcommand{\mD}{{\mathcal D}}

\newcommand{\mX}{{\mathcal X}}
\newcommand{\mY}{{\mathcal Y}}
\newcommand{\mL}{{\mathcal L}}
\newcommand{\mE}{{\mathcal E}}
\newcommand{\mZ}{{\mathcal Z}}

\DeclareMathOperator{\re}{Re}
\newcommand{\im}{\text{{\rm Im}}\,}
\newcommand{\sett}[2]{ \{ #1 \, \semic \, #2 \} }

\newcommand{\supp}{\text{{\rm supp}}\,}

\newcommand{\ran}{\textsf{R}}

\newcommand{\clos}[1]{\overline{#1}}

\newcommand{\sgn}{\text{{\rm sgn}}}
\newcommand{\barint}{\mbox{$ave \int$}}
\newcommand{\divv}{{\text{{\rm div}}}}
\newcommand{\curl}{{\text{{\rm curl}}}}
\newcommand{\esssup}{\text{{\rm ess sup}}}

\newcommand{\pd}{\partial}

\newcommand{\loc}{\text{{\rm loc}}}

\newcommand{\pv}{\text{p.v.}}
\newcommand{\ca}{\text{ca}}

\def\barint_#1{\mathchoice
            {\mathop{\vrule width 6pt
height 3 pt depth -2.5pt
                    \kern -8.8pt
\intop}\nolimits_{#1}}%
            {\mathop{\vrule width 5pt height
3 pt depth -2.6pt
                    \kern -6.5pt
\intop}\nolimits_{#1}}%
            {\mathop{\vrule width 5pt height
3 pt depth -2.6pt
                    \kern -6pt
\intop}\nolimits_{#1}}%
            {\mathop{\vrule width 5pt height
3 pt depth -2.6pt
          \kern -6pt \intop}\nolimits_{#1}}}

\usepackage{color}

\definecolor{gr}{rgb}   {0.,   0.8,   0. }
\definecolor{bl}{rgb}   {0.,   0.5,   1. }
\definecolor{mg}{rgb}   {0.7,  0.,    0.7}

\begin{document}

\title[Causal sparse domination of Beurling maximal regularity operators]
{Causal sparse domination of Beurling maximal regularity operators}
\author[Tuomas Hyt\"onen]{Tuomas Hyt\"onen$\,^1$}
\author[Andreas Ros\'en]{Andreas Ros\'en$\,^2$}
\thanks{$^1\,$T.H. was supported by the Academy of Finland (grant No. 314829 ``Frontiers of singular integrals'')}
\thanks{$^2\,$Formerly Andreas Axelsson}
\address{Tuomas Hyt\"onen, Department of Mathematics and Statistics, P.O. Box 68 (Pietari Kalmin katu 5), FI-00014 University of Helsinki, Finland} 
\email{tuomas.hytonen@helsinki.fi}
\address{Andreas Ros\'en\\Mathematical Sciences, Chalmers University of Technology and University of Gothenburg\\
SE-412 96 G{\"o}teborg, Sweden}
\email{andreas.rosen@chalmers.se}


\begin{abstract}
We prove boundedness of Calder\'on--Zygmund operators acting in Banach
functions spaces on domains, defined by the $L_1$ Carleson functional and
$L_q$ ($1<q<\infty$) Whitney averages. For such bounds to hold, we assume
that the operator maps towards the boundary of the domain.
We obtain the Carleson estimates by proving a pointwise domination of the
operator, by sparse operators with a causal structure.
The work is motivated by maximal regularity estimates for elliptic PDEs
and is related to one-sided weighted estimates for singular integrals.
\end{abstract}

\maketitle

\section{Introduction}

We prove estimates of Calder\'on--Zygmund operators (CZOs) acting on functions  
$f(t,x)$ defined on a half space 
$\R^{1+n}_+=\sett{\mv x=(t,x)}{t>0,x\in\R^n}$, in function 
spaces defined by $L_p(\R^n)$ norms of the classical Carleson and non-tangential 
maximal functionals $(Cf)(x)$ and $(Nf)(x)$ respectively.
We recall their definitions in \eqref{eq:Cdef} and \eqref{eq:Ndef} below.
For CZOs to be even locally bounded inside 
$\R^{1+n}_+$, it is necessary to modify $C$ and $N$ since 
these build on $L_1(\R^{1+n}_+)$ and $L_\infty(\R^{1+n}_+)$
norms respectively.
Following Kenig and Pipher~\cite{KenigPipher:93}, Auscher and 
Axelsson~\cite{AuscherAxelsson:11}, Hyt\"onen and 
Ros\'en~\cite{HytonenRosen:13} and Huang~\cite{Huang:16}, we use Whitney $L_q$ averaging $W_q f(t,x)$
with $1<q<\infty$, as defined in \eqref{eq:Wdef} below, and consider function
spaces defined by norms 
$$
\|C(W_{q} f)\|_{p}\quad\text{and} \quad
\|N(W_{q} f)\|_{p},
$$ 
which encode interior local $L_q$ regularity, transversal $L_1$ or $L_\infty$
regularity, and $L_p$ regularity along the boundary $\R^n$.
It was shown in \cite{HytonenRosen:13} that the norms
$\|N(W_{q'} f)\|_{p'}$ and
$\|C(W_{q} f)\|_{p}$ are dual, $1/p+1/p'=1$, $1/q+1/q'=1$.
In applications to elliptic partial differential equations (PDEs) as in 
\cite{KenigPipher:93, AuscherAxelsson:11},  
the case $q=2$ is particularly important, since for gradients $f$
of weak solutions to elliptic equations, in general we do not have 
pointwise interior bounds but only local $L_2$ estimates on
Whitney regions.

A general  CZO on $\R^{1+n}_+$ 
$$
  Sf(\mv x)= \int_{\R^{1+n}_+} k(\mv x,\mv y)f(\mv y) d\mv y,
$$
$\mv x= (t,x), \mv y=(s,y)$,
fails to be bounded in any of the norms above, even with Whitney averaging.
A key observation that we make is that for causal CZOs, bounds in some
of the norms hold.
By causal we mean that we consider either an upward mapping CZO, denoted
$S^+$, where $k=0$ for $t<s$, or a downward mapping CZO, denoted
$S^-$, where $k=0$ for $t>s$.
In our main Theorem~\ref{thm:mainthm}, we prove that downward mapping 
CZOs $S^-$ are bounded in Whitney modified Carleson norms $\|C(W_q f)\|_p$
and that, dually, upward mapping CZOs $S^+$ are bounded in Whitney modified 
non-tangential maximal norms $\|N(W_{q'} f)\|_{p'}$.

In Section~\ref{sec:PDEappl}, we discuss motivating examples: maximal
regularity operators arising from integration of elliptic PDEs 
in the variable transversal to the boundary as in 
\cite{AuscherAxelsson:11}.
At the end of the present paper, 
we explain how our Theorem~\ref{thm:mainthm}, in the
presence of pointwise kernel bounds, sharpens the estimate of half
of the maximal regularity operators in \cite{AuscherAxelsson:11}.
In the case of the Laplace equation in $\R^2_+$, the 
maximal regularity operators $S^\pm$ appearing are
the two causal parts of the Beurling transform, in
which case one of the key estimates from \cite{AuscherAxelsson:11}
independently have been proved by Astala and Gonz\'alez~\cite{AstalaGonzalez:16}.
We recall that the kernel $k(\mv x,\mv y)=(\mv x-\mv y)^{-2}$ of this Beurling transform is symmetric and as a consequence
its causal parts $S^\pm$ are $L_2$ bounded. 
Hence they are examples of rough CZOs: their kernels are discontinuous 
across the hypersurface $t=s$. With a slight abuse of notation we refer
to such rough CZOs simply as CZOs below.

Our Carleson estimate for $S^-$ uses the method of sparse domination 
of CZO by A. Lerner. In Section~\ref{sec:causaldom} 
we adapt the proof for general CZO from 
Lerner~\cite{Lerner:16} and prove a domination of causal CZOs by certain causal sparse
operators.  
Here we take a direction somewhat against the mainstream of recent activity, where the trend has been to establish sparse domination for ever more general classes of operators; in contrasts, we deal with operators having additional structure (the causality), and the point is to preserve at least part of this structure in the dominating sparse operators as well.
The one-dimensional special case, $n=0$, of our Theorem~\ref{thm:sparsecausal}
reads 
\begin{equation}   \label{eq:onedim}
  |S^- f(t)|\lesssim\sum_{t\in(a,a+\ell)\in\mD_f}\ell^{-1}\int_{a+\ell/2}^{a+2\ell} |f(s)|ds,
\end{equation}
where the sum is over a sparse collection $\mD_f$ of dyadic intervals $(a,a+\ell)$,
but the average is over the right half and a right neighbourhood of this interval.

There are a number of previous results dealing with causal (also known as one-sided) operators and related weight classes. Due to the intimate connections of sparse domination and weighted norm inequalities, which have been explored in the recent literature, we briefly comment on these works. 
Sawyer \cite{Sawyer:86} found that the one-sided maximal operator on the line satisfies weighted estimates if and only if the weight satisfies a one-sided $A_p$ condition. Later, Aimar et al. \cite{AFMR:97} showed that the one-sided $A_p$ condition is also the right condition for the boundedness of one-sided singular integrals on the line. Versions of these results for operators on a half-line (case $n=0$ of our setting in $\R_+^{1+n}$), including extensions to operator-valued kernels and applications to maximal regularity of Cauchy problems, are recently due to Chill and Krol \cite{ChillKrol:18}. We refer to this paper for an extensive bibliography of other related works.

A sharp form of the one-sided maximal function estimates was found by Mart\'{\i}n-Reyes and de la Torre \cite{MRdlT:15}. The corresponding problem (a ``one-sided $A_2$ conjecture'') for one-sided singular integrals remains open, but the analogous result for one-sided martingale transforms has been achieved in Chen and Lacey \cite{ChenLacey:20}. At the time of writing, the established methods of reducing continuous singular integrals to their dyadic models (including \cite{Lacey:elem,Lerner:simple,Lerner:16,LerNaz:19,LerOmb:20} in the case of regular CZOs and \cite{BFP:16,CACDPO:17} for some rough extensions, to name but a selection of the extensive literature) are not available in a form that would respect the one-sided structure. In the present work we obtain a result that may be seen as a version of sparse domination for one-sided operators. While this version serves our present purposes, it is not strong enough to make progress on the mentioned weighted questions. The problem is that our dominating sparse operators are not strictly one-sided (for that, the integral in \eqref{eq:onedim} should be over 
$(a+\ell, a+2\ell)$ only) and cannot possibly be estimated in terms of one-sided $A_p$ weights.

As mentioned above, one-sided weighted estimates for vector-valued singular integrals have been applied to maximal regularity operators for Cauchy problems
 in \cite{ChillKrol:18}.
However, our Carleson estimates cannot be viewed as such weighted estimates, but
rather correspond to the end-point norms $r=1$ and $r=\infty$ in the scale of $L_r$ tent spaces from Coifman, Meyer and Stein~\cite{CoifmanMeyerStein:85}.
Estimates for maximal regularity operators in tent spaces based on the $L_2$ 
area functional are in 
Auscher, Kriegler, Monniaux and Portal~\cite{AuscherKrieglerMonniauxPortal:12}. 
Tent space estimates of horizontally 
mapping,  that is, acting in the $\R^n$ variable,  CZOs are in Auscher and Prisuelos-Arribas~\cite{AuscherPrisuelos:17}.
In Section~\ref{sec:counterex}, we include counterexamples that show 
that for non-causal $\R^{1+n}_+$ CZOs, as well as horizontally mapping
$\R^n$ CZOs, the Carleson estimates considered in this paper fail in general.

\section{CZOs on  non-tangential  and Carleson spaces}  \label{sec:counterex}

We denote cubes on the boundary $\R^n$ by $Q,R,\ldots,$
and cubes in the half-space $\R^{1+n}_+$ by boldface $\mv Q, \mv R,\ldots$.
We denote $n$ or $n+1$ dimensional measure by $|\cdot|$ depending
on the dimension of the cube, and we denote by $cQ$, $c>0$, the cube with same center
as $Q$ and sidelength $\ell(cQ)= c\ell(Q)$.
For $\mv Q\subset \R^{1+n}_+$ we form $c\mv Q$ in $\R^{1+n}$ as above,
and define $c\mv Q\subset \R^{1+n}_+$ as the intersection with $\R^{1+n}_+$,
which may not be a cube if $\mv Q$ is near $\R^n$. 
A Carleson cube in $\R^{1+n}_+$ is a cube of the form
$$
  \mv Q= Q^\ca = (0,\ell(Q))\times Q,
$$
where $Q$ is a cube in $\R^n$ with sidelength $\ell(Q)$.
The corresponding Whitney region is the top half
$$
  Q^w= (\ell(Q)/2, \ell(Q))\times Q
$$
of $Q^\ca$, and is the union of $2^n$ cubes in $\R^{1+n}_+$ of 
sidelength $\ell(Q)/2$.
We denote points in $\R^n$ by $x,y,\ldots,$ and 
points in $\R^{1+n}_+$ by boldface $\mv x=(t,x), \mv y=(s,t),\ldots$
Indicator functions of cubes and more general sets $E$ are denoted by
$1_E$. 
Subsets, not necessarily strict, are denoted $A\subset B$. 

In particular we use dyadic cubes, and  
 let $\mD=\bigcup_{j\in\Z}\mD_j$ denote a system of dyadic cubes in $\R^n$,
 with $\mD_j$ being the cubes of sidelength $\ell(Q)= 2^{-j}$, such that
 the dyadic cubes in $\mD$ form a connected tree under inclusion.
Let $\mv\mD=\bigcup_{j\in\Z}\mv\mD_j$ denote the associated dyadic
system for $\R^{1+n}_+$, where $\mv\mD_j$ consists of dyadic cubes
of the form $\mv Q= (k2^{-j}, (k+1)2^{-j})\times Q$, $Q\in\mD_j$, $k=0,1,2,\ldots$.
We note that also $\mv\mD$ form a connected tree under inclusion where
the dyadic Carleson cubes, the cubes touching $\R^n$, will play the 
special role.

We consider CZOs acting on functions defined on the half space $\R^{1+n}_+$,
$n\ge 1$,  and belonging to endpoint spaces in the scale 
of tent spaces from \cite{CoifmanMeyerStein:85}, whose
definitions use 
the non-tangential maximal and Carleson functionals 
\begin{align}
  Nf(x) &= \esssup_{|y-x|<\alpha t} |f(t,y)|, \label{eq:Ndef} \\
  Cf(x) &= \sup_{Q\ni x}|Q|^{-1}\iint_{Q^\ca}|f(t,x)|dtdx, \label{eq:Cdef}
\end{align}
where the second supremum is over all cubes $Q\subset\R^n$ containing
$x$,
and the Whitney averaging operator
\begin{equation}  \label{eq:Wdef}
  W_q f(t,x)= |Q^w|^{-1/q} \|f\|_{L_q(Q^w)}, 
\end{equation}
where $Q\subset \R^n$ is the cube with center $x$ and sidelength $t$.
More precisely, we consider the norms  $\|N(W_{q'} f)\|_{L_{p'}(\R^n)}$,
introduced in \cite{KenigPipher:93}, and predual norms
$\|C(W_{q} f)\|_{L_{p}(\R^n)}$  considered in \cite{HytonenRosen:13}. 
We exclude $p=1$, where the former norms are $\infty$ except
for $f=0$, while the 
latter norms are $\|W_{q} f\|_\infty$. 
To avoid technicalities we also exclude $p=\infty$, although
this endpoint case sometimes is of interest and some of our
results hold for $p=\infty$, mutatis mutandis.  
We recall from \cite{HytonenRosen:13,HytonenRosen:18} that up to 
constants these two scales of norms are independent of the choice
of aperture $\alpha>0$ and the precise shape of the Whitney regions
and Carleson cubes. Moreover, for the dyadic functionals
\begin{align*}
  N_{\mD}f(x) &= \sup_{Q\in\mD, Q\ni x} f_{Q}, \\
  C_{\mD}f(x) &= \sup_{Q\in \mD, Q\ni x}|Q|^{-1}\sum_{R\in\mD,R\subset Q}
  |R^w| f_{R},
\end{align*}
acting on sequences $(f_Q)_{Q\in\mD}$,
and dyadic Whitney $L_q$ averages 
$$
(W_{\mD,q}f)_Q= |Q^w|^{-1/q} \|f\|_{L_q(Q^w)},  \qquad Q\in\mD,
$$ 
we have  $\|N_{\mD}W_{\mD, q'}f\|_{p'}\approx \|N(W_{q'} f)\|_{p'}$ and
$\|C_{\mD}W_{\mD,q}f\|_{p}\approx \|C(W_{q} f)\|_{p}$ for
$1< p,q <\infty$. 
Finally, we recall that if $1<p<\infty$, then $\|Cf\|_p\approx \|Af\|_p$,
where
$$
  Af(x)= \iint_{|y-x|<\alpha t} |f(t,y)|t^{-n} dtdx,
$$
but the estimate $\gtrsim$ breaks down as $p\to\infty$,
and $\lesssim$ breaks down as $p\to 1$.

We are interested in boundedness of CZOs in the norms
$\|C(W_q f)\|_p$ or equivalently, by duality,
in the norms $\|N(W_{q'} f)\|_{p'}$. 
As discussed in the Introduction, it is necessary to require 
$1<q<\infty$, and main applications concern $q=2$.
For the Carleson functional, which builds on $L_1$ integrals, it is
natural to make use of the recent technology of sparse domination
of CZOs.
We recall the following estimate of A. Lerner~\cite{Lerner:16}.
Assume given a standard CZO $S$ on $\R^{1+n}_+$ and a function 
$f\in L_1(\R^{1+n}_+)$
with bounded support.
Then there exists a collection $\mv\mD_f\subset\mv\mD$ of cubes,
which is $\eta$-sparse, for some $0<\eta<1$ independent of $f$,
such that 
$$
  |Sf(\mv x)| \lesssim \sum_{\mv Q\in\mv \mD_f, Q\ni \mv x} 
  \barint_{3\mv Q} |f(\mv y)| d\mv y
$$
holds pointwise almost everywere.
Here $\eta$-sparse means that 
there exists subsets $E_{\mv Q}\subset \mv Q$, $\mv Q\in \mv \mD_f$,
such that $|E_{\mv Q}|\ge \eta |\mv Q|$ for all $\mv Q\in \mv \mD_f$, and 
$E_{\mv Q}\cap E_{\mv R}=\emptyset$ if $\mv Q\ne \mv R$.

The method of sparse domination has proven very successful in proving
optimal bounds for CZOs on weighted $L_p$ spaces.
We apply in this paper the technique to the scales of Banach
function spaces described above.
Since $N$ and $C$ involve an $L_\infty$ and $L_1$ norm respectively however,
it is not surprising that in general we have the following counterexamples.

\begin{ex}  \label{ex:counterup}
  A sparse operator, that is a sublinear operator of the form
$$
  \widehat Sf(\mv x) = \sum_{\mv Q\in\widehat{\mv\mD}, \mv Q\ni \mv x} 
  \barint_{3\mv Q} |f(\mv y)| d\mv y,
$$
with a fixed $\eta$-sparse collection $\widehat{\mv\mD}$ independent of $f$, is not in general bounded on $\|C(W_{q} f)\|_p$ for any $p,q$.
An example is as follows.
Let $n=1$ and fix a  dyadic  Carleson cube $\mv Q_0= Q_0^\ca$ with $\ell(Q_0)=1$.
Consider
\begin{equation}  \label{eq:fNfcn}
  f_N= 2^N 1_{(0,2^{-N})\times Q_0}
\end{equation}
and let $\widehat{\mv \mD}$ be the collection of all dyadic Carleson cubes contained in 
$\mv Q_0$.
Then $W_{\mD,q} f_N=f_N$  since $f_N$ is constant on Whitney regions,  $C_{\mD}f_N= 1$ on $Q_0$  since 
$\sup_{x\in Q_0}\int_0^1 f_N dt=1$,
and $\|C_{\mD}f_N\|_p\lesssim 1$ for any fixed $p>1$.
However, for $\mv x=(t,x)\in Q^w$ with $Q\subset Q_0$, $\ell(Q)= 2^{-k}$, $0\le k\le N$, we have
$$
  \widehat S f_N(\mv x)= \sum_{j=0}^k (2^N 2^{-N}2^{-j})/(2^{-j})^2
  = 2^{k+1}-1\approx 1/t.
$$
Hence $C_{\mD}(\widehat Sf_N)\gtrsim \int_{2^{-N}}^1 \frac{dt}{t}\approx N$
at each point in $Q_0$. 
It follows that $\widehat S$ fails to be bounded in the norm $\|C(W_q f)\|_p$
for any 
$1<p,q<\infty$, since $q\to W_q$ is increasing. 
\end{ex}

\begin{ex}   \label{ex:upBeucounter}
 Although domination  of CZOs by sparse operators is considered to be quite sharp, 
 Example~\ref{ex:counterup}  does not disprove that CZOs are bounded in these
Carleson norms. And indeed, by direct calculations one can show that the
Beurling transform on the upper half plane
\begin{equation}   \label{eq:Beurling}
  Sf(z)= \pv\frac{-1}\pi \int_{\im w>0} \frac {f(w)}{(w-z)^2} |dw|
\end{equation}
maps $f_N$  from \eqref{eq:fNfcn}  boundedly in the Carleson norms.
To find a counterexample for the Beurling transform, it is convenient 
to consider a weak limit 
\begin{equation}   \label{eq:weaklim}
  f(t,x)= g(x)\delta_0(t),
\end{equation}
of functions like $f_N$, to be modulated by a function $g(x)$ supported on 
$Q_0=(0,1)$ that remains to be chosen. 
In this limit we have
$$
   S f(z)= \pv\frac{-1}\pi \int_0^1 \frac {g(x)}{(x-z)^2} dx,
$$
and for real-valued $g$ we note from Cauchy--Riemann's equations
that $|Sf(z)|=|\nabla u(z)|$, where $u$ is the Poisson extension of $g$.
As discussed in  \cite[Intro.]{HytonenRosen:18},  by contructing
$g$ via a lacunary Fourier series, it is known that the
bound
\begin{equation}   \label{eq:Cup<gp}
  \|C(\nabla u)\|_p\lesssim \|g\|_p
\end{equation}
cannot hold uniformly for $g$, for any fixed $1<p<\infty$. 
Replacing $\delta_0$ in \eqref{eq:weaklim} by $2^N 1_{(0,2^{-N})}$, it  can be shown 
that $S$ fails to be bounded in the norm $\|C(W_q f)\|_p$
for any 
$1<p,q<\infty$. 
 The technical details of this counterexample are found in the 
Appendix of this paper. 
\end{ex}

\begin{ex} 
We next demonstrate that $\R^n$ CZOs acting horizontally, that is pointwise in $t$, in general are not bounded either in the norms 
$\|N(W_{q'} f)\|_{p'}$ or $\|C(W_q f)\|_p$.
A concrete counterexample is the following.
Let 
$$
  Hf(t,x)= \frac 1\pi \pv\int_\R \frac{f(t,y)}{x-y}dy
$$
be the Hilbert transform acting horizontally in $\R^2_+$, 
and consider the function
$$
  f_N(t,x)= \sum_{j=0}^N\sum_{k=2^j}^{2^{j+1}-1} 2^k 1_{(2^{-k}, 2^{1-k})}(t)
  1_{(k2^{-j}-1, (k+1)2^{-j}-1)}(x)
$$
with support in $[0,1]^2$. 
The terms in this double sum are supported on thin rectangles, with integrals
equal to their lengths in the $x$-direction.
Consider a Carleson square $\mv Q$ with sidelength $2^{-m}$.
The terms for a given $0\le j\le N$ will have in total at most integral $2^{-m}$ inside
$\mv Q$, so $C(W_{\mD,\infty} f_N)\lesssim N$ on $[0,1]$, since the thin rectangles
are unions of Whitney regions.
We conclude that
$\|C(W_\infty f_N)\|_p\lesssim N$ for each $p>1$,
since $f_N$ is supported on $[0,1]^2$. 
However, an explicit computation reveals that
$$
  \int_0^1  |Hf_N(t,x)| dx \approx  \int_0^1 2^k \ln(1+\tfrac {2^{-j}}x) dx \approx 2^{k-j}(j+1)
$$
for $t\in (2^{-k}, 2^{1-k})$, $2^j\le k<2^{j+1}$, $0\le j\le N$. 
Therefore
$$
  \int_{[0,1]^2} |Hf_N| dtdx\approx \sum_{j=0}^N\sum_{k=2^j}^{2^{j+1}-1} 2^{-j}(j+1)\approx \sum_{j=0}^N (j+1)\approx N^2,
$$
so
$C(Hf_N)\gtrsim N^2$ on $(0,1)$.
It follows that $H$ fails to be bounded in the norm $\|C(W_q f)\|_p$
for any 
$1<p,q<\infty$, since $q\to W_q$ is increasing.  
\end{ex}

\section{Causal Calder\'on--Zygmund operators}   \label{sec:causalSIOs}

We consider a Calder\'on--Zygmund (CZ) kernel $k(\mv x,\mv y)$ on 
$\R^{1+n}_+$. 
More precisely, we assume kernel bounds 
\begin{equation}   \label{eq:kernelbound}
  |k(\mv x,\mv y)|\lesssim |\mv x-\mv y|^{-(n+1)}
\end{equation}
and regularity 
\begin{equation}   \label{eq:kernelreg}
  \max(|k(\mv x,\mv y+\mv t))-k(\mv x,\mv y)|, |k(\mv x+\mv t,\mv y)-k(\mv x,\mv y)|)\lesssim |\mv t|^\gamma/|\mv x-\mv y|^{n+1+\gamma},
\end{equation}
for all $|\mv t|\le |\mv x-\mv y|/2$ and a fixed $0<\gamma\le 1$.

We consider a linear operator operator $S^+$ which is bounded
$L_2(\R^{1+n}_+)\to L_2(\R^{1+n}_+)$ and has rough CZ kernel
$$
  k^+((t, x), (s,y))=
  \begin{cases}
    k((t, x), (s,y)), & t>s, \\
    0, & t<s.
  \end{cases}
$$ 
We also consider a linear operator operator $S^-$ which is bounded
$L_2(\R^{1+n}_+)\to L_2(\R^{1+n}_+)$ and has rough CZ kernel
$$
  k^-((t, x), (s,y))=
  \begin{cases}
    0, & t>s, \\
    k((t, x), (s,y)), & t<s.
  \end{cases}
$$ 
Thus $S^\pm$ are (rough) CZOs, where 
$S^+$ is upward mapping away from
$\R^n$ and $S^-$ is downward mapping towards $\R^n$.
We refer to $S^\pm$ as {\em causal operators}.
They are simple examples of singular integral with rough kernels,
in that $k^\pm((t,x),(s,y))$ may be discontinuous on the hyperplane $t=s$
in $\R^{2(1+n)}$, with a simple jump discontinuity when $x\ne y$ and $t=s$.

\begin{lem}   \label{lem:hormander}
  The  kernels $k^\pm$  satisfy the H\"ormander regularity condition
$$
  \int_{(3\mv Q)^c} (|k^\pm(\mv x,\mv y_1))-k^\pm(\mv x,\mv y_2)| + |k^\pm(\mv y_1,\mv x)-k^\pm(\mv y_2,\mv x)|) d\mv x \lesssim 1,
$$
uniformly for all cubes $\mv Q\subset \R^{1+n}_+$ and $\mv y_1,\mv y_2\in \mv Q$.
\end{lem}

\begin{proof}
  For $\mv Q= (a,b)\times Q$, this follows from using \eqref{eq:kernelreg}
  for $t<a$ and $t>b$, and using \eqref{eq:kernelbound} for
  $a<t<b$.
\end{proof}

Denote by $Mf$ the standard Hardy--Littlewood maximal function of $f$.
As in \cite{Lerner:16}, we also require a weak $L_1$ estimate of Lerner's maximal singular integral
$$
  M_{S^\pm} f(\mv x)= \sup_{\mv Q\ni \mv x} \| S^\pm(1_{(3\mv Q)^c}f)\|_{L_\infty(\mv Q)},
  \qquad \mv x\in \R^{1+n}_+,
$$
with supremum over dyadic cubes $\mv Q\subset \R^{1+n}_+$, into which 
$S^\pm$ maps from the complement of the enlarged (non-dyadic) cube
$3\mv Q$. 

\begin{prop}   \label{prop:weakL1}
  The causal CZOs $S^\pm$ and the maximal singular integrals $M_{S^\pm}$ 
  are all bounded from 
  $L_1(\R^{1+n}_+)$ to $L_{1,\infty}(\R^{1+n}_+)$.
\end{prop} 

\begin{proof}
(1) The weak $L_1$ estimate for $S^\pm$ itself follows from the standard
proof employing the Calder\'on--Zygmund decomposition, since this only 
requires the H\"ormander regularity estimate from Lemma~\ref{lem:hormander}.

(2) Following the standard proof of Cotlar's lemma, see for example
\cite[Sec. 7.7]{CoifmanMeyer2:97}, we write 
$$
  f_1= 1_{3\mv Q}f\quad\text{and}\quad f_2= 1_{(3\mv Q)^c}f,
$$
for a cube $\mv Q$ with center $\mv x_0$.
We estimate $S^\pm(1_{(3\mv Q)^c}f)= S^\pm f_2$ at each fixed $\mv x_1\in \mv Q$, by
comparing it to the value at a variable point $\mv x\in\mv  Q$, writing
$$
  S^\pm f_2(\mv x_1)= (S^\pm f_2(\mv x_1)-S^\pm f_2(\mv x))+ S^\pm f(\mv x)- S^\pm f_1(\mv x)
  =: I+II+III.
$$
Raising the terms to power $1/p$, $p>1$, and taking the average over $\mv Q$,
terms $II$ and $III$ are estimated in the usual way:
$\barint_{\mv Q} |S^\pm f|^{1/p}  d\mv x  \lesssim M(|S^\pm f|^{1/p})(\mv x_0)$, and
$\barint_{\mv Q} |S^\pm f_1|^{1/p}  d\mv x \lesssim |\mv Q|^{-1/p} \|S^\pm f_1\|_{L_{1,\infty}}^{1/p} \lesssim |\mv Q|^{-1/p}\|f_1\|_{L_1}^{1/p}\lesssim (Mf(\mv x_0))^{1/p}$
by Kolmogorov's inequality and weak $L_1$ estimate (1) above.

For $I$, we define
$$
  \Phi(\mv z)=\Phi(u,z)= 
  \begin{cases}
    (1+|\mv z|)^{-(1+n+\gamma)}, & u\in(-\infty,-1)\cup (1,\infty),\\
    (1+|\mv z|)^{-(1+n)}, & u\in(-1,1).
  \end{cases}
$$
Using \eqref{eq:kernelreg} and \eqref{eq:kernelbound}
$$
  |(S^\pm f_2(\mv x_1)-S^\pm f_2(\mv x))|\lesssim M_\Phi f(\mv x_0),
$$
where $M_\Phi$ is the maximal operator
$$
  M_\Phi f= \sup_{t>0}( |f|* t^{-(n+1)}\Phi(\cdot /t) ). 
$$
By \cite[Sec. II.4, Prop. 2]{SteinHA:93}, $M_\Phi$ is weak-type (1,1).
This completes the proof, since we have
$$
  |M_{S^\pm}f(\mv x_0)|\lesssim M_\Phi f(\mv x_0)+ M(|S^\pm f|^{1/p})(\mv x_0)^p
  + Mf(\mv x_0)
$$
and $M$ is bounded on $L_{p,\infty}$ and weak-type (1,1).
\end{proof}

\section{Causal sparse domination of $S^-$}   \label{sec:causaldom}

To state our causal sparse domination,
we need the following subsets of the neighbourhood $3\mv Q$.
For $\mv Q= (a,a+\ell)\times Q\in \mv\mD$, $Q\in\mD$, we define
the upper and lower halves of $\mv Q$
\begin{align*}
  \mv Q^u &= (a+\ell/2, a+\ell)\times Q,\\
  \mv Q^l &= (a,a+\ell/2)\times Q, 
\end{align*}
the parts of  $(a,a+2\ell)\times 3Q$  above and below $t= a+\ell/2$ 
\begin{align*}
    \mv Q^T &= (a+\ell/2, a+2\ell)\times 3Q,\\
  \mv Q^L &= (a,a+\ell/2)\times 3Q, 
\end{align*}  
and  the part of $3\mv Q\setminus \clos{\mv Q^l}$ above $t=a$ 
$$
  \mv Q^\sqcap =((a,a+2\ell)\times 3Q) \setminus \clos{\mv Q^l}.
$$
For $\mv x=(t,x)\in\mv Q$ we let
$$
  \mv Q^\sqcap_{\mv x}=
   \mv Q^\sqcap\cap \sett{(s,y)}{s>\min(a+\ell/2,t)},
$$
so that $\mv Q^T\subset \mv Q^\sqcap_{\mv x}\subset \mv Q^\sqcap$.
The following sparse estimate is a causal adaptation of the estimate
in \cite{Lerner:16}.
We specifically note that the very argument of \cite{Lerner:16} seemed more amenable to this adaptation than either its predecessors or successors in the sparse domination literature.

\begin{thm}  \label{thm:sparsecausal}
Let $S^-$ be a downward mapping causal Calder\'on--Zygmund operator
as in Section~\ref{sec:causalSIOs}.
Let $f\in L_1(\R^{1+n})$ have bounded support.
Then there exists a $1/4$-sparse family of cubes $\mv\mD_f\subset \mv \mD$
such that 
$$
  |S^-f(\mv x)|\lesssim\sum_{\mv Q\in \mv\mD_f, \mv Q\ni \mv x} \barint_{\mv Q^\sqcap_{\mv x}}|f(\mv y)| d\mv y,
  \qquad\text{for a.e. } \mv x\in \R^{1+n}_+.
$$
\end{thm}

We remark that an essential point in Theorem~\ref{thm:sparsecausal} is that the
sum uses
the same dyadic system as we start out with.
If we allow us to replace $\mv\mD$ by a finite number of other dyadic
systems, then the result is immediate from \cite{Lerner:16}, 
but cannot be used to prove our Carleson bounds.

We also remark that the one-dimensional result, $n=0$, is somewhat cleaner.
In this case $\mv\mD$ is the standard dyadic intervals 
and for $\mv Q=(a,a+\ell)$, the average in the sparse sum
is over 
$(a+\ell/2, a+2\ell)$, independent of $\mv x\in\mv Q$, as in \eqref{eq:onedim}. 
The reason why $Q^T$, which is the direct analogue of 
$(a+\ell/2, a+2\ell)$ for $n\ge 1$, must be replaced by 
the larger set $\mv Q^\sqcap_{\mv x}$ in 
Theorem~\ref{thm:sparsecausal}, is the term $II_1$ appearing
in its proof.
The even larger set $\mv Q^\sqcap$ is too large, since it 
cannot control the term $II_1$ appearing in the later proof 
of  Theorem~\ref{thm:mainthm}: it is important that these
sets do not touch $t=a$.
However, we have found that Theorem~\ref{thm:mainthm} is valid
when $\mv Q^\sqcap_{\mv x}$ is enlarged to a set,
to be denoted by $\mv Q^\Box_{\mv x}$, which goes
down to the $t$ coordinate of $\mv x$ also inside $\mv Q$.

\begin{proof}
(1)
Fix $\mv Q\in\mv\mD$ and assume that the closure of $3\mv Q$ contains the support of $f$.
Write $f^T:= 1_{\mv Q^T}f$  and define the set 
\begin{multline*}
  E:=\sett{\mv x\in\R^{1+n}_+}{|S^-f^T(\mv x)|> c\barint_{\mv Q^T} |f|d\mv y}\\
  \cup
  \sett{\mv x\in\R^{1+n}_+}{|M_{S^-}f^T(\mv x)|> c\barint_{\mv Q^T} |f|d\mv y}.
\end{multline*}
Let $\mv R_1, \ldots, \mv R_{2^n}$ denote  all the children 
of $\mv Q$ (maximal subcubes strictly contained in $\mv Q$) 
contained in $\mv Q^l$, and let $\mv Q_1, \mv Q_2, \ldots$ be an enumeration of all the maximal subcubes of $\mv Q^u$ such that 
\begin{equation}   \label{eq:stoppingcond}
  |\mv Q_j\cap E|>\alpha |\mv Q_j|.
\end{equation}
Here $c>0$ and $\alpha\in(0,1)$ are parameters to be fixed below.
We obtain a family $\{\mv R_i, \mv Q_j\}$ of disjoint subcubes of $\mv Q$  that give us a disjoint union
\begin{equation*}
  \mv Q=\mv Q^l\cup\mv Q^u=\Big(\bigcup_i \mv R_i\Big)\cup\Big(\bigcup_j \mv Q_j\Big)\cup (\mv Q^u\setminus\bigcup_j \mv Q_j),
\end{equation*}
Hence, splitting the indicator $1_{\mv Q}$ outside $S^-$ in the first step and the function $f$ inside $S^-$ in the second one,
\begin{equation*}
\begin{split}
  1_{\mv Q}S^- f
  &=\sum_i 1_{\mv R_i}S^- f+\sum_j 1_{\mv Q_j}S^- f+1_{\mv Q^u\setminus \bigcup \mv Q_j}S^-f  \\
  &=\sum_i 1_{\mv R_i}S^- (1_{3\mv R_i}f)+\sum_j 1_{\mv Q_j}S^- (1_{3\mv Q_j}f)+ \\
  &\qquad+\sum_i 1_{\mv R_i}S^- (1_{3\mv Q\setminus 3\mv R_i}f)
  + 1_{\mv Q^u\setminus \bigcup \mv Q_j}S^- f^T +\sum_j 1_{\mv Q_j}S^- (1_{3\mv Q\setminus 3\mv Q_j}f^T) \\
  &=:I_1+I_2+ II_1+II_2+II_3.
\end{split}
\end{equation*}
Using causality for $S^-$, we have replaced $f$ by $f^T$ in terms $II_2$ and $II_3$. 

(2)
We next show that terms $II$ are bounded by
$\barint_{\mv Q^\sqcap_{\mv x}}|f|d\mv y$, 
for almost every  $\mv x\in\mv Q$.
For $II_1$, this follows from \eqref{eq:kernelbound}
and the downward mapping property of $S^-$  since
$|\mv x-\mv y|^{-(1+n)}\lesssim |\mv Q_{\mv x}^\sqcap|^{-1}$
for all $\mv y\in (3\mv Q)\setminus(3\mv R_i)$ and $\mv x\in\mv R_i$. 
For $II_2$, it follows from Lebegue's differentiation theorem 
that 
$$
  E\cap \mv Q^u\subset\bigcup_j \mv Q_j
$$
modulo a set of measure zero.
The estimate  of $II_2$ then follows from the definition of $E$. 

For $II_3$, consider one of the stopping cubes $\mv Q_j$.
By definition, we have $|\mv Q_j\cap E|>\alpha |\mv Q_j|$.
We require an upper estimate of the measure of $E$.
It follows from Proposition~\ref{prop:weakL1}
that 
$$
  |E|\lesssim \left(c\barint_{\mv Q^T} |f|d\mv y\right)^{-1} \int_{\R^{1+n}_+} |f^T|d\mv y
  \approx |Q|/c.
$$
Moreover, if $\mv Q_j^p$ denotes the dyadic parent of $\mv Q_j$ 
(the minimal dyadic cube such that $\mv Q_j^p\supsetneqq \mv Q_j$),
 then by maximality of $\mv Q_j$,  its parent does not satisfy
\eqref{eq:stoppingcond} and  we have
$$
|\mv Q_j\cap E|\le |\mv Q_j^p\cap E|\le \alpha |\mv Q_j^p|=
\alpha 2^{1+n} |\mv Q_j|.
$$
Let $\alpha= 2^{-2-n}$. Then in particular there exists 
$\mv x\in \mv Q_j\setminus E$  since $|\mv Q_j\setminus E|\ge |\mv Q_j|/2>0$ with this choice of $\alpha$. 
From the definition of $E$ it now follows that
$$
  |S^- (1_{(3\mv Q)\setminus(3\mv Q_j)} f^T)|
  \le M_{S^-}f^T (\mv x)\lesssim  \barint_{\mv Q^T} |f|d\mv y
$$
on almost all of $\mv Q_j$. 
This proves the estimate of $II_3$. 

(3)
We have shown that for $f$ supported on $3\mv Q$, we have
$$
 |1_{\mv Q} S^- f- \sum_i 1_{\mv R_i} S^- (1_{3\mv R_i}f)
  - \sum_j 1_{\mv Q_j} S^- (1_{3\mv Q_j} f)|\lesssim  \barint_{\mv Q^\sqcap_{\mv x}} |f|d\mv y
$$
at almost all $\mv x\in\mv Q$.
To conclude the proof, we iterate this estimate for all
the subcubes $\mv R_i$ and $\mv Q_j$.
To this end, we note that 
$$
  \sum_i |\mv R_i|= |\mv Q|/2
$$
and, since $\bigcup \mv Q_j\subset\{M(1_E)>\alpha\}$, that
$$
\sum_j|\mv Q_j|\lesssim \alpha^{-1} \int 1_E d\mv y= |E|/\alpha\lesssim |\mv Q|/(c\alpha).
$$
Fix $c$ large enough so that $\sum_i |\mv R_i|+\sum_j |\mv Q_j|\le 3|Q|/4$.
Define 
$$
  E_{\mv Q}:= \mv Q\setminus\left( \bigcup_i \mv R_i\cup \bigcup_j \mv Q_j  \right),
$$
so that $|E_{\mv Q}|\ge |\mv Q|/4$.

Finally consider a disjoint union
$$
  \R^{1+n}_+ = \bigcup_k \mv Q^k,
$$ 
modulo zero sets, where $\mv Q^k\in\mv \mD$ are constructed as follows.
Since $\supp(f)$ is bounded and $\mD$ is connected, we can choose
$\mv Q^1\in\mv\mD$ containing $\supp(f)$.
Then let $\mv P^1= \mv Q^1$ and recursively define 
$\mv P^{j+1}$ to be the dyadic parent of $\mv P^{j}$,
that is, the smallest dyadic cube such that 
$\mv P^{j+1}\supsetneqq \mv P^{j}$.
The siblings of $\mv P^{j}$ are the other dyadic subcubes of $\mv P^{j+1}$ of same size as $\mv P^{j}$.
Now choose $\mv Q^2, \mv Q^3, \ldots$ to be an ordering of all the siblings 
of all $\mv P^j$, $j=1,2,\ldots$. 
Then $\mv Q^1\subset 3\mv Q^k$ for all $k$, so that
$$
  S^-f= \sum_k 1_{\mv Q^k} S^-f= \sum_k 1_{\mv Q^k} S^-(1_{3\mv Q^k} f),
$$ 
and the estimates above apply to each $\mv Q= \mv Q^k$. 
We define the family of dyadic cubes $\mv \mD_f$ to be 
$\mv Q^1, \mv Q^2, \ldots$ along with, for each of $\mv Q^k$, all generations
of  stopping cubes $Q_j$ and subcubes $R_i$  starting from $\mv Q=\mv Q^k$, constructed as  in (1)  above.
It follows that $\mv \mD_f$ is $1/4$ sparse and that we have the stated
sparse bound of $S^-f$.
\end{proof}

\section{Bounds for causal CZOs}  \label{sec:boundofsparse}

The Carleson bounds established in this section apply to slightly
larger causal sparse operators
\begin{equation}  \label{eq:largersparse}
  |\widehat S f(\mv x)|\lesssim\sum_{\mv Q\in\widehat{\mv\mD}, \mv Q\ni \mv x} \barint_{\mv Q^\Box_{\mv x}}|f(\mv y)| d\mv y,
  \qquad \mv x\in \R^{1+n}_+,
\end{equation}
where we have replaced the region $\mv Q^\sqcap_{\mv x}$
appearing in the sparse operator in Theorem~\ref{thm:sparsecausal},
by the slightly larger region 
$$
  \mv Q^\Box_{\mv x}=
  (\min(a+\ell/2,t), a+2\ell)\times 3Q,
$$ 
for $\mv x=(t,x)\in\mv Q=(a,a+\ell)\times Q$.
We also write
$$
   \mv Q^\Box=(a, a+2\ell)\times 3Q
$$ 
for the union of $Q^T$ and $Q^L$, modulo a set of measure zero. 
The following theorem is the main result of this paper. 
Its proof uses regions $\mv R^\Box$ for 
dyadic cubes $\mv R\subset Q^w$, where $Q^w$
is a dyadic Whitney region.
We define enlarged auxiliary Whitney regions
$$
\widetilde Q^w=(\ell(Q)/2,3\ell(Q)/2)\times (2Q),
$$
so that $\mv R^\Box\subset \widetilde Q^w$
whenever $\mv R\subset Q^w$. 

\begin{thm}     \label{thm:mainthm}
Let $S^+$ and $S^-$ be causal Calder\'on--Zygmund operators
as in Section~\ref{sec:causalSIOs}
  and  $1<p,q< \infty$. 
Then we have the estimates
\begin{align}
  \|CW_q(S^-f)\|_p &\lesssim \|CW_q(f)\|_p, \label{eq:pqestSmin} \\
  \|NW_{q'}(S^+f)\|_{p'} &\lesssim \|NW_{q'}(f)\|_{p'}. \label{eq:pqestSplus}
\end{align} 
More precisely, the estimate
\begin{equation}   \label{eq:CWforShat}
  C_{\mD}W_{\mD,q} (\widehat S f)\lesssim C_\mD \widetilde W_{\mD,q}(f)+ C(f)
\end{equation}
holds pointwise on $\R^n$ for the causal sparse operator $\widehat S$
from \eqref{eq:largersparse}, for any $\eta$-sparse collection
$\widehat {\mv\mD}\subset\mv\mD$ with $0<\eta<1$.
Here  $\widetilde W_{\mD,q}(f)$ denotes the Whitney $L_q$ average 
over $\widetilde Q^w$. 
\end{thm}

 We remark that interior regularity $1<q<\infty$ is necessary 
for boundedness of CZOs, and is used to ensure boundedness of the 
Hardy--Littlewood maximal operator in the proof below, 
while $p>1$ is necessary for the 
Carleson and non-tangential maximal norms to be non-degenerate.
The restriction $p<\infty$ is less essential, but is made so that the
following lemma hold.

\begin{lem}   \label{lem:densedual}
Fix $1<p,q<\infty$ and define Banach spaces
\begin{align*}
  Z &= \sett{f\in L_q^\loc(\R^{1+n}_+)}{\|CW_q(f)\|_p<\infty} 
  \quad\text{and}\\
  X &= \sett{f\in L_{q'}^\loc(\R^{1+n}_+)}{\|NW_{q'}(f)\|_{p'}<\infty}.
\end{align*}
Let $Z_0$ denote the subspace of functions 
$f\in L_q(\R^{1+n}_+)\cap Z$ with
bounded support, and let 
$X_0$ denote the subspace of functions 
$f\in L_{q'}(\R^{1+n}_+)\cap X$ with
bounded support.

Then $Z_0$ is dense in $Z$, in the norm $\|CW_q(\cdot)\|_p$, and
$X_0$ is dense in $X$, in the norm $\|NW_{q'}(\cdot)\|_{p'}$.
Moreover, $Z$ is non-reflexive and $X$ equals the dual space $Z^*$, 
using the standard $L_2(\R^{1+n}_+)$ pairing.
\end{lem}

\begin{proof}
The density of $Z_0$ in $Z$ follows from 
\cite[Lem. 2.5]{HytonenRosen:13}, which proves that even 
compactly supported $L_q$ functions are dense in $Z$.
Also the non-reflexivity of $Z$ and 
that $X= Z^*$ was shown in \cite[Thm. 3.2]{HytonenRosen:13}.

To see the density of $X_0$ in $X$, let $\epsilon>0$ and 
$f\in X$ so that 
$\|N_{\mD}W_{\mD, q'}f\|_{p'}<\infty$.
Recall that equivalences of norms allow us to use the dyadic
versions of the norms.
Since $p>1$ and $\mD$ is connected, there exists $Q_0\in\mD$
such that 
$\int_{\R^n\setminus Q_0}|N_{\mD}W_{\mD, q'}f|^{p'}d\mv x\le \epsilon^{p'}$.
Let $\mD_f$ be the set of maximal cubes $Q\subset Q_0$ such that
$$
  (W_{\mD,q'}f)_Q>1/\epsilon.
$$
Define functions 
$$
f_0=\sum_{Q\in\mD_f} f1_{Q^\ca}\quad \text{and}\quad 
f_1= f1_{\R^{1+n}_+\setminus Q_0^\ca},
$$
and let $f_2= f-f_0-f_1$.
We have $N_{\mD}W_{\mD, q'}f_0= 
1_E N_{\mD}W_{\mD, q'}(f1_{Q_0^\ca})$, 
where $E= \sett{x\in \R^n}{N_{\mD}W_{\mD, q'}(f1_{Q_0^\ca})>\epsilon^{-1}}$.
Since $|E|\to 0$, it follows from dominated convergence
that $\|f_0\|_X\to 0$ as $\epsilon\to 0$.

To estimate $f_1$, let $Q_1$ denote the dyadic parent of $Q_0$.
Then
$$
  \sup_{Q_0} N_{\mD}W_{\mD, q'}f_1
  \le
  \inf_{Q_1\setminus Q_0} N_{\mD}W_{\mD, q'}f,
$$
and we conclude that 
$\|f_1\|_X\lesssim \epsilon$
since $N_{\mD}W_{\mD, q'}f_1=N_{\mD}W_{\mD, q'}f$
on $\R^n\setminus Q_0$.
Since $f=f_0+f_1+f_2$ where $\supp f_2\subset Q_0$
and
$$
  \int_{\R^{1+n}_+} |f_2|^{q'} d\mv x
  = \sum_{Q\subset Q_0,
  Q\not\subset\bigcup\mD_f} \int_{Q^w}|f|^{q'} d\mv x
  \le \sum_{Q\subset Q_0, Q\not\subset\bigcup\mD_f} \epsilon^{-q'} |Q^w|
  \le  \epsilon^{-q'} |Q_0^\ca|<\infty,
$$ 
it follows that $f_2\in X_0$. This shows that $X_0$ is dense in $X$.
\end{proof}

\begin{proof}[Proof of Theorem~\ref{thm:mainthm}]
 We first show that it suffices to prove the estimate \eqref{eq:CWforShat}.
Indeed, combining this and Theorem~\ref{thm:sparsecausal} yields
the bound for $S^-$ on $Z_0$, using 
the notation from Lemma~\ref{lem:densedual} and
equivalences of norms between the
different versions of Carleson functionals and Whitney averages
from \cite[Sec. 3]{HytonenRosen:13}. 
By Lemma~\ref{lem:densedual}, we have 
a unique extension of $S^-$ by continuity to all $Z$.

Next consider $S^+$ on $L_{q'}(\R^{1+n}_+)\supset X_0$, with adjoint
$(S^+)^*$ on $L_{q}(\R^{1+n}_+)\supset Z_0$.
For $f\in X_0$ and $g\in Z_0$, we obtain the bound
\begin{equation}   \label{eq:fgdualityest}
  \int_{\R^{1+n}_+} |S^+f(\mv x)| |g(\mv x)| d\mv x
  \lesssim \|f\|_X \|g\|_Z
\end{equation}
by applying the duality to $f$ and $g\,\sgn((S^+f) g)$.
By Fatou's lemma and the density of $Z_0$ in $Z$, \eqref{eq:fgdualityest}
continues to hold for all $g\in Z$.
Let $h= S^+ f$ and let $K\subset \R^{1+n}_+$
be a compact set, so that $\|NW_{q'}(h 1_K)\|_{p'}<\infty$.
It follows from \cite[Thm 3.2]{HytonenRosen:13} and
\eqref{eq:fgdualityest} that 
$\|h 1_K\|_X\lesssim \|f\|_X$.
Exhausting $\R^{1+n}_+$ with $K$ and using Fatou's lemma
shows that $S^+ f\in X$ with $\|S^+ f\|_X\lesssim \|f\|_X$,
for all $f\in X_0$.
By Lemma~\ref{lem:densedual}, we have 
a unique extension of $S^+$ by continuity to all $X$.

It remains to prove \eqref{eq:CWforShat}. 
Fix $Q_0\in\mD$ and consider $\int_{Q_0^\ca} |W_{\mD,q}(\widehat Sf)|d\mv x$.
Consider a Whitney region $Q^w\subset Q_0^\ca$, where $Q\subset Q_0$,
$Q\in\mD$.
We write
$$
\widehat S f(\mv x)=
\sum_{\mv R\in \widehat {\mv\mD}, \mv R\ni \mv x} \barint_{\mv R^\Box_{\mv x}}|f|d\mv y 
= I+II+III,
$$
for a.e. $\mv x\in Q^w$,
where we split the sum according to the cases $\mv R\subsetneqq Q^\ca$,
$\mv R\supset Q^\ca$ and $\mv R\cap Q^\ca=\emptyset$,
noting that $\mv R, Q^\ca\in\mv\mD$,  with $Q^w$ being the top half of $Q^\ca$. 
Clearly $III=0$, for $I$  we have $\mv R\subset Q^w$
since $\mv R\ni \mv x$,  and for $II$ 
we have that $\mv R= R^\ca$ are Carleson boxes  since $Q^\ca$,
and hence $\mv R$, touches $\R^n$. 

The local terms $I$ do not require causality and we may replace $\mv R^\Box_{\mv x}$
by the larger set $\mv R^\Box$. Standard sparse estimates 
via duality apply as follows.
Let $g\ge 0$ be supported on $Q^w$ with $\int g^{q'}d\mv x=1$, $1/q+1/q'=1$.
Then
\begin{multline*}
  \int_{Q^w} g\left(\sum_{\mv R\in\widehat{\mv\mD}, \mv R\subset Q^w}
   1_{\mv R}\barint_{\mv R^\Box}|f|d\mv y\right) d\mv x   
  =\sum_{\mv R\in \widehat{\mv\mD}, \mv R\subset Q^w}
  \left(\barint_{\mv R} g\, d\mv x\right)
   \left(\barint_{\mv R^\Box}|f|d\mv y\right)   |\mv R|\\
 \lesssim \sum_{\mv R\in \widehat{\mv\mD}, \mv R\subset Q^w} \int_{E_R}
  \left(\barint_{\mv R} g\, d\mv x\right)
   \left(\barint_{\mv R^\Box}|f|d\mv y\right)  d\mv z \\
   \lesssim 
   \int_{Q^w}M(g) M(1_{\widetilde Q^w}f) d\mv z 
   \lesssim |Q^w|^{1/q} \left(\barint_{\widetilde Q^w}|f|^q d\mv z \right)^{1/q}.
\end{multline*}
We have used that $|\mv R|\lesssim |E_{\mv R}|$, that
$\barint_{\mv R} g d\mv x\le M(g)$ and 
$\barint_{\mv R^\Box}|f|d\mv y\le M(1_{\widetilde Q^w}f)$
on all $E_{\mv R}$,
H\"older's inequality, and $L_{q'}$ and $L_q$ boundedness of
the Hardy--Littlewood maximal operator $M$. 
Taking supremum over $g$ gives the pointwise bound
$W_{\mD,q}(I)\lesssim \widetilde W_{\mD,q}(f)$. 
Integration over $Q_0^\ca$ yields the desired estimate of $I$
by $C_\mD \widetilde W_{\mD,q}(f)$. 

The non-local terms $II$ do not require sparseness, and we may increase
the sum over
$R^\ca\in\widehat{\mv\mD}$ to a sum over $R\in \mD$.
Note that the term $II$ is essentially constant on $Q^w$, and
we replace the Whitney $L_q$ average by the supremum norm on $Q^w$,
and increase $(R^\ca)^\Box_{\mv x}$ to 
$$
(R^\ca)^\Box_{Q}:= (R^\ca)^\Box\cap
\sett{(s,y)}{s>\ell(Q)/2}.
$$ 
We have
\begin{multline*}
  \int_{Q_0^\ca} |II|d\mv x\lesssim \sum_{Q\in\mD, Q\subset Q_0} |Q^w|
  \sum_{R\in\mD, R\supset Q} |R^\ca|^{-1} \int_{(R^\ca)^\Box_Q}|f|d\mv y  \\
 = \sum_{R\in\mD}\left( \sum_{Q\in\mD, Q\subset Q_0, Q\subset R} |Q^w| |R^\ca|^{-1} \int_{(R^\ca)^\Box_Q}|f|d\mv y\right) = II_1+II_2, 
\end{multline*}
where we split the outer sum according to $R\subset Q_0$ and
$R\supsetneqq Q_0$.
For the tail terms $II_2$ we increase $(R^\ca)^\Box_Q\subset (R^\ca)^\Box= 3R^\ca$,  where we recall that 
$3R^\ca=3R^\ca\cap \R^{1+n}_+$ according to our definition, 
and get
$$
  II_2\le |Q_0^\ca| \sum_{R\in\mD, R\supsetneqq Q_0}  |R^\ca|^{-1} \int_{3R^\ca}|f|d\mv y
  \lesssim |Q_0|\left( \sum_{k=1}^\infty 2^{-k}\right)\inf_{Q_0} C f,
$$ 
since $|3R|/|R^\ca|\lesssim 2^{-k}/\ell(Q_0)$ for the $k$'th ancestor of $Q_0$, and since $|Q_0^\ca|/\ell(Q_0)=|Q_0|$.
For the mid-range terms $II_1$, we estimate
\begin{multline*}
  II_1= \sum_{R\in\mD, R\subset Q_0} |R^\ca|^{-1}\int_{3R^\ca} \left( \sum_{Q\in\mD, Q\subset R} |Q^w| 1_{\ell(Q)<2s} \right) |f(s,y)| dsdy \\
  \lesssim 
   \sum_{R\in\mD, R\subset Q_0} |R^\ca|^{-1}\int_{3R^\ca} (s|R|) |f(s,y)|dsdy \\
   \lesssim 
   \int_{3Q_0^\ca} \left(  \sum_{R\in\mD, R\subset Q_0} s\ell(R)^{-1} 1_{3R^\ca}  \right) |f(s,y)|dsdy \\ 
   \lesssim \int_{3Q_0^\ca} \left( s \sum_{2^{-k}\ge s/3} (2^{-k})^{-1}  \right) |f(s,y)|dsdy 
   \lesssim \int_{3Q_0^\ca} |f|dsdy.
\end{multline*} 
For the second last estimate, note that at a point $(s,y)\in 3Q_0^\ca$
we have
$1_{3R^\ca}=0$ if $2^{-k}<s/3$, where $\ell(R)=2^{-k}$, 
while at each scale $2^{-k}\ge s/3$ at most $3^n$ cubes $R$
contribute the the sum. 
In total we obtain a bound of $II$ by $C(f)$, which completes the proof.
\end{proof}

\section{Motivations from elliptic PDEs} \label{sec:PDEappl}

We end this paper by discussing motivations for the estimates
in Theorem~\ref{thm:mainthm}  coming from
maximal regularity estimates  for PDEs.
A natural point of departure for the discussion is the parabolic problem 
$$
  \pd_t u_t+Lu_t=g_t,
$$
for $u_t(x)= u(t,x)$, $t>0$, $x\in\R^n$, with $u_0=0$.
Here $L= -\divv_x A(x)\nabla_x$, where $A\in L_\infty(\R^n,\mL(\C^{n}))$ is accretive.
The maximal regularity problem for a space $\mH$ of functions in $\R^{1+n}_+$, 
is whether a source $g\in\mH$ yields a solution $u$ with $L u$ 
 (or equivalently $\pd_t u$)  in $\mH$.
The  parabolic maximal regularity operator  is
\begin{equation}   \label{eq:parabmaxreg}
  Lu_t= \int_0^t L e^{-(t-s)L} g_s ds,
\end{equation}
which is upward mapping like the operators $S^+$ considered in this paper.
However,  the operator $g\mapsto Lu$  does not have the CZO kernel bounds
\eqref{eq:kernelbound}, even for $L=-\Delta_x$.
Indeed  the estimate
$|\pd_t (t^{-n/2}e^{-|x|^2/(4t)})| \lesssim (t+|x|)^{-(1+n)}$ fails,
and  non-tangential estimates are not natural for this problem due to  the parabolic scaling $t\sim |x|^2$. 

Instead our motivation comes from the analogous considerations for
elliptic divergence form equations
\begin{equation}   \label{eq:secondorder}
  \divv_{t,x} A(t,x) \nabla_{t,x} u (t,x) =0, \qquad t>0, x\in\R^n,
\end{equation}
where 
$A=\begin{bmatrix} a & b \\ c & d \end{bmatrix}\in 
L_\infty(\R^{1+n}_+,\mL(\C^{1+n}))$ 
is accretive.
Following \cite{AuscherAxelssonMcIntosh:10,AuscherAxelsson:11}, we consider the
associated first order generalized Cauchy--Riemann system
\begin{equation}   \label{eq:firstorder}
  \pd_t f_t+ DB_t f_t = 0
\end{equation}
for the conormal gradient 
$f_t= \begin{bmatrix} a\pd_t u+ b\nabla_{x}u &
\nabla_x u \end{bmatrix}^T$  of $u$,  consisting of the 
conormal derivative and tangential gradient. 
This uses the  self-adjoint first order differential  operator
$D= \begin{bmatrix} 0 & \divv_x \\ -\nabla_x & 0 \end{bmatrix}$
and transformed accretive matrix
$$B= \begin{bmatrix} a^{-1} & -a^{-1}b \\ ca^{-1} & d-ca^{-1}b \end{bmatrix},$$ 
acting in $L_2(\R^n; \C^{1+n})$ for each $t>0$. 
For each $t>0$, the conormal gradient belongs to the closure of the range
of $D$ since $\curl_x(\nabla_x u)=0$, and
it was demonstrated in \cite[Sec. 3]{AuscherAxelssonMcIntosh:10}
that \eqref{eq:secondorder} is equivalent to \eqref{eq:firstorder}, under
this tangential curl free constraint on the tangential part of $f_t$.

Now fix bounded and accretive coefficients $B= B(x)$, which are independent
of $t$, and consider the equation
\begin{equation}   \label{eq:firstordermaxreg}
  \pd_t f_t+ DB f_t = g_t,
\end{equation}
where $f$ and $g$ are functions of $t>0$ and take values in 
$\mH= \clos{\ran(D)}\subset L_2(\R^n; \C^{1+n})$. 
The $t$-independent case when $B_t=B$ and $g=0$ was studied in 
\cite{AuscherAxelssonMcIntosh:10}, whereas perturbation results for
$t$-dependent coefficients
$B_t$ were obtained in \cite{AuscherAxelsson:11} via Duhamel's principle
with $g_t= D(B-B_t)f_t$. 
As in these works, we make use of the holomorphic functional calculus of
the bisectorial operator $DB$ in $\mH$:
the symbols $\chi^\pm(z)=\begin{cases} 1, & \pm\re z>0,\\ 0, &\pm\re z<0, \end{cases}$ yield the spectral projections 
$E^\pm= \chi^\pm(DB)$,
the symbol $|z|= z(\chi^+(z)-\chi^-(z))$ yields
the sectorial operator $\Lambda= |DB|$
and the symbols $e^{-t|z|}$
yield the Poisson semigroup $e^{-t\Lambda}$
generated by $\Lambda$. 

Following \cite{AuscherAxelsson:11}, we integrate \eqref{eq:firstordermaxreg}
with boundary and decay conditions 
$\lim_{t\to 0} E^+ f_t=0= \lim_{t\to\infty} E^- f_t$ 
and a general source term $g_t$, and obtain 
\begin{equation}   \label{eq:implicitdefSpm}
  -\pd_t f_t= DB f_t 
  = T^+ g_t+ T^- g_t,
\end{equation}
which uses both an upward mapping maximal regularity operator 
\begin{equation}   \label{eq:upmaxregDB}
  T^+ g_t=  \int_0^t \Lambda e^{-(t-s)\Lambda}E^+  g_s ds
\end{equation}
and a downward mapping maximal regularity operator
\begin{equation}  \label{eq:downmaxregDB}
  T^- g_t= \int_t^\infty\Lambda e^{-(s-t)\Lambda} E^- g_s ds.
\end{equation}  
We saw above that the parabolic maximal regularity problem concerned
an upward mapping operator  \eqref{eq:parabmaxreg},  using the 
heat semigroup $\exp(-tL)$, which is not of the form considered in
this paper.
The elliptic maximal regularity problem involves  both upward and downward mapping
operators $T^+$ and $T^-$. Both $T^\pm$  use the Poisson 
semigroup $\exp(-t\Lambda)$,
and  $T^\pm$ generalize the causal parts of the  
Beurling singular integral
as the following example from \cite{NystromRosen:14} shows.

\begin{ex}   \label{ex:cauchylip}
  Consider the Cauchy--Riemann equations for a holomorphic function $f(z)$
  in the region above the graph $\gamma$ of a Lipschitz function $y= \phi(x)$.
  Parametrizing $z= x+i(\phi(x)+t)$, these equations read
$$
  \pd_t f+ BDf= 0  \quad\text{in } \R^{2}_+ 
$$  
when written in the real basis $\{1+i \phi'(x), i\}$.
   Here $B= (1+i\phi'(x))^{-1}$ is an accretive, $t$-independent complex multiplier,
   and $D= -i\pd_x$ is the first order self-adjoint derivative.
   
   Consider now the associated Poisson semigroups 
   and maximal regularity operators  in \eqref{eq:implicitdefSpm},
   with $DB$ replaced  by the similar operator $BD= B(DB)B^{-1}$. 
   As observed in \cite[Lem. 3.1]{NystromRosen:14},
   by computing the kernels of resolvents and operators in the functional
   calculus of $BD$, the causal operators from 
   \eqref{eq:upmaxregDB} and \eqref{eq:downmaxregDB}
   are
   \begin{equation}   \label{eq:beurlingint+}
     T^+g_t(z)
     = \frac{-1}{2\pi} \int_0^t \int_\gamma\frac{g_s(w)dw}{(w+is-(z+it))^2}ds
   \end{equation}
   and
   \begin{equation}   \label{eq:beurlingint-}
     T^-g_t(z)
     = \frac{-1}{2\pi} \int_t^\infty \int_\gamma\frac{g_s(w)dw}{(w+is-(z+it))^2}ds.
   \end{equation} 
   We parametrize $w= y+i\phi(y)$, $y\in\R$, and consider the 
   bounded and invertible multiplier $dw/dy= 1+i\phi'(y)$.
   It follows that $T^\pm= S^\pm (1+i\phi')$, where $S^\pm$ are
   defined as in \eqref{eq:beurlingint+} and \eqref{eq:beurlingint-},
   but replacing $dw$ by $dy$.
   
   The operators $T^\pm$, and hence also the operators $S^\pm$,
   are bounded on $L_2(\R)$. See \cite[Thm. 6.5]{AuscherAxelsson:11},
   which also applies to the more general operators
   \eqref{eq:upmaxregDB} and \eqref{eq:downmaxregDB}
   in any dimension.
   Moreover, we see that the kernels $k^\pm$ for $S^\pm$ are the
   causal truncations of
   $$
     k((t,x),(s,y))= \frac 1{(y+i(\phi(y)+s)-x-i(\phi(x)+t))^2}.
   $$
   Since $\phi$ is a Lipschitz function, it is readily verified that
   the CZ  bounds \eqref{eq:kernelbound} and \eqref{eq:kernelreg} hold
   for this $k$, with $\gamma=1$. 
\end{ex}

For general elliptic equations \eqref{eq:secondorder}, the estimate \eqref{eq:kernelbound} of the distribution kernels of the associated operators  $T^\pm$  holds only in an average $L_2$ sense.
Global $L_2$ bounds 
$$
  \| \Lambda e^{-|t-s|\Lambda}E^\pm \|_{ L_2\to L_2 } \lesssim 1/|t-s|
$$
follow from the Kato quadratic estimates, more precisely
\cite[Thm. 3.1]{AxelssonKeithMcIntosh:06}.
Also $L_2$ off-diagonal estimates for $\Lambda e^{-|t-s|\Lambda}E^\pm$
can be derived from such estimates for the resolvents of $DB$
(see \cite[Prop. 5.2]{AxelssonKeithMcIntosh:06}).
As for pointwise kernel estimates, we saw in Example~\ref{ex:cauchylip}
that for general non-smooth coefficients $B$, only
the first estimate in \eqref{eq:kernelreg} can hold for $DB$, 
and  similarly  only the second estimate can hold for $BD$.
More importantly, even \eqref{eq:kernelbound} may fail for $n\ge 2$.
For $n=1$, at least for real coefficients, the pointwise
kernel bound \eqref{eq:kernelbound} follows from the interior regularity
estimates for solutions to \eqref{eq:secondorder} of
De Giorgi, Nash and Moser (which hold for real coefficients in any dimension).
Indeed, $f_t= e^{-t|BD|} f_0$ solves the first order equation
$\pd_t f_t + BD f_t=0$, where the first component of $f$ satisfies
a second order equation \eqref{eq:secondorder} and hence have 
pointwise bounds.
See \cite[Sec. 3.2]{AuscherAxelsson:11}.
In general this is not true for the remaining components (conjugate functions) 
in $f$, except when $n=1$, in which case also the second component/conjugate function does satisfy an equation \eqref{eq:secondorder},  with  some 
conjugate coefficients.
See \cite[Lem. 5.3]{AuscherRosen:12}.
To summarize, the operators  $T^\pm$ from
\eqref{eq:upmaxregDB} and \eqref{eq:downmaxregDB} 
are causal CZOs, as in Section~\ref{sec:causalSIOs}, at least if $n=1$,
if $A$ is real and if we factor out an invertible multiplicative factor.

Coming to estimates of  $T^\pm$,  
we consider conormal gradients $f$ solving the 
Cauchy--Riemann system \eqref{eq:firstorder} in $\R^{1+n}_+$.
It is known from \cite{AuscherAxelsson:11} 
that on the one hand the function space $\mX$ in $\R^{1+n}_+$ with norm
$$
  \|f\|_\mX=\|N(W_2 f)\|_2,
$$
along with the subspace $\mY^*= L_2(\R^{1+n}_+; t^{-1}dtdx)$,
are natural for boundary value problems with boundary topology $L_2(\R^n)$
for $\lim_{t\to 0} f_t$.
On the other hand, the function space
$$
  \mY= L_2(\R^{1+n}_+; tdtdx),
$$
along with the subspace $\mZ$ with norm $\|f\|_{\mZ}=\|C(W_2 f)\|_2$,
are natural for boundary value problems with boundary topology $H^{-1}(\R^n)$
for $\lim_{t\to 0} f_t$. (Corresponding to boundary function space $L_2(\R^n)$
for the potential $u$.)
It was shown in \cite[Thm. 3.1]{HytonenRosen:13} that $\mZ^*=\mX$,
but that $\mZ$ is not reflexive.

Previously known bounds for the causal operators  $T^\pm$ in 
\eqref{eq:upmaxregDB} and \eqref{eq:downmaxregDB}  from 
\cite[Prop. 7.1]{AuscherAxelsson:11}, 
for general accretive coefficients $B$ and not assuming
pointwise kernel bounds, can be summarized in the following
diagrams.
$$
\xymatrix@1{
  \mX 
  \ar[d]_{\mE} & \mX \ar[d]^{\mE}  \\
  \mY^*\ar[ru]|->>>>{T^-} \ar[r]^>>>{T^+} & \mY^*
  }
  \qquad\qquad\qquad
  \xymatrix@1{
  \mY \ar[r]^>>>{T^-} \ar[d]_{\mE} & \mY \ar[d]^{\mE}  \\
  \mZ\ar[ru]|->>>>{T^+} 
  & \mZ
  }
$$
The estimate 
\begin{equation}  \label{eq:L1L2tentspace}
 \left(  \int_{\R^2_+} |f(t,y)|^2 t dtdx \right)^{1/2} \approx
  \|(A(|f|^2 t))^{1/2}\|_2 \lesssim   \|A(W_2 f)\|_2 \approx
  \|C(W_2 f)\|_2
\end{equation}
shows that $\mZ\subset \mY$, and hence $\mY^*\subset \mX$.
It was shown in \cite[Lem. 5.5]{AuscherAxelsson:11} that
a multiplication operator $\mE$ maps $\mX\to \mY^*$ if 
it satisfies the Carleson condition 
\begin{equation}   \label{eq:CarlcondmE}
\|C(W_\infty\mE^2/t)\|_\infty<\infty.
\end{equation}
Moreover, it follows from \cite{HytonenRosen:13} that $\mE$
maps as in the diagrams if and only if \eqref{eq:CarlcondmE} holds.

We solve PDEs $\pd_t f_t+ DB_t f_t=0$ with $t$-dependent
coefficients by freezing the coefficients $B_t$ to $B$.
Writing
$$
  \mE_t= B^{-1}(B-B_t),
$$
with $B= \lim_{t\to 0}B_t$, the PDE becomes
$\pd_t f_t+ DB f_t=DB\mE_t$. 
To prove the existence of boundary values of solutions $f_t$
and representation formulas as
in \cite[Sec. 8-9]{AuscherAxelsson:11}, it is needed that 
\begin{equation}   \label{eq:SpmmE}
  (T^+ + T^-)\mE
\end{equation} 
is a bounded operator (with small norm to obtain a Cauchy type 
representation formula  of $f$ in terms of the boundary 
trace $f_0$). 
This will yield equivalences of norms 
$\|f_0\|_{L_2(\R^n)}\approx \|f\|_\mX$ for solutions $f\in \mX$ to
\eqref{eq:firstorder}, and
$\|f_0\|_{H^{-1}(\R^n)}\approx \|f\|_\mY$ for solutions $f\in \mY$ to
\eqref{eq:firstorder}. 
Thus two fundamental estimates in the study of boundary problems 
for elliptic divergence form equations \eqref{eq:secondorder}
are to bound the operators $T^\pm\mE$ in the 
$\mX$ and $\mY$ norms, under as weak as possible hypothesis on $\mE$, which measures the $t$-variation of the coefficients. 

\begin{itemize}
\item
{\bf $L_2(\R^n)$ boundary traces of solutions $f$:}
Assuming that pointwise kernel bounds hold as discussed above, 
so that $T^+$ is a CZO  modulo a trivial multiplicative factor, 
it follows from Theorem~\ref{thm:mainthm} that already  $T^+:\mX\to\mX$ 
is bounded. We obtain the new result that only $\mE\in L_\infty(\R^{1+n}_+)$,
which is weaker than \eqref{eq:CarlcondmE},
is needed for boundedness of  $T^+ \mE:\mX\to\mX$. 
However, this does not yield any new estimate for 
 $(T^++T^-)\mE$ as a whole,  needed
for the PDE application.
Indeed, for the downward mapping maximal regularity operator, the best known
bound is that  $T^-: \mY^*\to \mX$. 
Similar to below,
we conjecture that the Carleson condition \eqref{eq:CarlcondmE} is in general necessary for the boundedness of  $T^-\mE: \mX\to \mX$. 

\item
{\bf $H^{-1}(\R^n)$ boundary traces of solutions $f$:}   
It is known from \cite{AuscherAxelsson:11,HytonenRosen:13}
that already  $T^-:\mY\to\mY$  is bounded, so 
only $\mE\in L_\infty(\R^{1+n}_+)$ is needed for  $T^- \mE:\mY\to\mY$. 
The new result from Theorem~\ref{thm:mainthm} 
that  $T^-:\mZ\to \mZ$  is bounded, is not relevant for the
boundedness of  $T^- \mE:\mY\to\mY$. 

For the upward mapping maximal regularity operator, the best known
bound is that  $T^+: \mZ\to \mY$. 
This follows by duality from \cite[Thm. 6.8]{AuscherAxelsson:11}.
Furthermore, the Carleson condition \eqref{eq:CarlcondmE} on multipliers $\mE$ 
is in general necessary for the boundedness of  $T^+ \mE:\mY\to\mY$. 
To see this, by duality and the above estimate  $T^-:\mY\to\mY$,  we equivalently consider the boundedness of  
$\mE^*(T^++T^-):\mY^*\to\mY^*$. 
In the case of the Beurling transform it was proved in
Astala-Gonzalez \cite[Thm. 1]{AstalaGonzalez:16}
that this latter boundedness holds if and only if $\mE$ satisfies
\eqref{eq:CarlcondmE}, but without the Whitney factor $W_\infty$ (due
to the pointwise kernel bounds present for the Beurling transform).
\end{itemize}

\subsection{Summary and open problems}

Causal maximal regularity operators $T^\pm$ are fundamental 
in the study of elliptic boundary value problems, as evidenced by
\cite{AuscherAxelsson:11}.
Such operators $T^\pm$ are close to being causal CZOs $S^\pm$
and we have investigated boundedness of such $S^\pm$ in this paper.
In regards to this motivation from elliptic PDEs, our results are
meager: assuming pointwise kernel bounds, it follows from
Theorem~\ref{thm:mainthm} that the $\mX$ norm 
of $T^+\mE$ is bounded by
$\|\mE\|_{\infty}$. 
However, we believe that the techniques developed in this paper
are important and that future work in this direction may result
in more substantial applications.
We end by formulating some open problems.

\begin{itemize}
\item
Can the estimates in Theorem~\ref{thm:mainthm}, with $p=q=2$
at least, be shown for the more general operators $T^\pm$ appearing in 
\eqref{eq:upmaxregDB} and \eqref{eq:downmaxregDB},
without assuming pointwise kernel bounds?
A positive result in this direction are the weighted norm estimates
through sparse domination
for such generalized CZOs by 
Bernicot, Frey and Petermichl~\cite{BFP:16}.

\item
Consider the $L_r$ based tent spaces
of Coifman, Meyer and Stein~\cite{CoifmanMeyerStein:85},
with $1<r<\infty$ and in particular $r=2$.
The present paper concerns the end point cases $r=1$ and $r=\infty$,
and the duality results \cite{HytonenRosen:13} have been extended to the scale of tent spaces by Huang~\cite{Huang:16}.
Can sparse estimates, possibly using $L_r$ averages,
prove tent space estimates of CZOs,
or more general maximal regularity operators as in \cite{AuscherKrieglerMonniauxPortal:12}?
A concrete PDE application would be to reprove the bound 
$T^+:\mZ\to \mY$ from \cite[Prop. 7.1]{AuscherAxelsson:11},
using sparse estimates.
Note that $\mZ$ and $\mY$ are $L_1$ and $L_2$ based tent spaces
respectively.

\item
We started this Section by considering the more well known 
upward mapping
maximal regularity operator \eqref{eq:parabmaxreg}
for parabolic initial value problems, but this does not have
the CZ kernel bounds. 
Can sparse estimates be adapted to show a parabolic  
analogue of the non-tangential bound
\eqref{eq:pqestSplus}?
\end{itemize}

\appendix

\section{The Beurling counterexample}

We supply the details omitted in Example~\ref{ex:upBeucounter},
which show that no Carleson bound can hold for the Beurling transform
$S$ defined in \eqref{eq:Beurling}.
Assuming for contradiction that
\begin{equation}\label{eq:SbdCWqp-prelim}
  \|C(W_q(Sf))\|_p\lesssim\|C(W_q f)\|_p,
\end{equation}
we derive a series of consequences, eventually leading to an obvious impossibility. By \cite[Prop. 3.7]{HytonenRosen:13}, we may replace $C$ and $W_q$ by their dyadic versions, but we continue to denote them by the same symbols for simplicity. Since $W_q(Sf)\geq W_1(Sf)$ and $C(W_1(Sf))=C(Sf)$, our first derived bound will be
\begin{equation}\label{eq:SbdCWqp}
  \|C(Sf)\|_p\lesssim\|C(W_q f)\|_p.
\end{equation}
As suggested in Example~\ref{ex:upBeucounter}, we apply \eqref{eq:SbdCWqp} to $f_N(t,x)=g(x)\phi_N(t)$, where $\phi_N=2^N 1_{(0,2^{-N})}$, and investigate the limit $N\to\infty$.

We first consider the right-hand side. For a Whitney cube $R^w$, we have
\begin{equation*}
  |R^w|^{-1/q}\|f_N\|_{L_q(R^w)}=
  \begin{cases} 2^N |R|^{-1/q}\|g\|_{L_q(R)}, & \text{if }\ell(R)\leq 2^{-N}, \\ 0, & \text{else}.\end{cases}
\end{equation*}
Let us denote by
\begin{equation*}
  M_{q,2^{-N}}g(x):=\sup_{\substack{ R:x\in R\in\mathcal D \\ \ell(R)\leq 2^{-N}}}|R|^{-1/q}\|g\|_{L_q(R)}
\end{equation*}
the $L_q$ (dyadic) maximal operator restricted to cubes of side-length $\ell(R)\leq 2^{-N}$.
Thus, for all $x\in Q\in\mathcal D$, we have
\begin{equation*}
  \frac{1}{|Q|}\int_{Q^{\operatorname{ca}}}W_q f_N
  \leq \frac{1}{|Q|}\int_0^{\min(\ell(Q),2^{-N})}\int_{Q} 2^N M_{q,2^{-N}} g(y)dy\,dt\leq M(M_{q,2^{-N}} g)(x),
\end{equation*}
and hence $C(W_q f_N)\leq M(M_{q,2^{-N}} g)$ and
\begin{equation*}
  \|C(W_q f_N)\|_p
  \leq \|M(M_{q,2^{-N}} g)\|_p
  \lesssim \|M_{q,2^{-N}}g\|_p,\quad p>1.
\end{equation*}
If $p>q$, we could further dominate this by $\|g\|_p$, uniformly in $N$, but we do not wish to impose this restriction. Instead, if $g$ is, say, a continous compactly supported function (and this will be enough for our eventual counterexample), then $M_{q,2^{-N}}g(x)\to |g(x)|$ pointwise, as $N\to\infty$, and $M_{q,2^{-N}}g$ is dominated by a bounded and compactly supported function, uniformly in $N$. It hence follows from dominated convergence that
\begin{equation}\label{eq:RHS}
  \limsup_{N\to\infty}\|C(W_q f_N)\|_p\lesssim\|g\|_p.
\end{equation}

We then turn to the left-hand side of \eqref{eq:SbdCWqp}. If $\im z>2^{-N}$, there is no singularity in \eqref{eq:Beurling} for $f=f_N$, and hence
\begin{equation*}
\begin{split}
  Sf_N(z) &= \frac{-1}\pi \int_{\im w>0} \frac {f_N(w)}{(w-z)^2} |dw|,\qquad \im z>2^{-N}, \\
  &=2^N\int_0^{2^{-N}}\Big(\frac{-1}\pi \int_{\R} \frac {g(s)}{(s+it-z)^2}ds\Big)dt
  \underset{N\to\infty}\longrightarrow \frac{-1}\pi \int_{\R} \frac {g(s)}{(s-z)^2}ds.
\end{split}
\end{equation*}
Since any $z\in\C$ with $\im z>0$ satisfies $\im z>2^{-N}$ when $N$ is large enough, we have
\begin{equation*}
  \lim_{N\to\infty}Sf_N(z)= \frac{-1}\pi \int_{\R} \frac {g(s)}{(s-z)^2}ds=:v(z),\qquad \im z>0.
\end{equation*}
From repeated applications of Fatou's lemma it then follows that
\begin{equation*}
\begin{split}
  \frac{1}{|Q|}\int_{Q^{\operatorname{ca}}}|v(z)|\,|dz|
  &\leq\liminf_{N\to\infty}\frac{1}{|Q|}\int_{Q^{\operatorname{ca}}}|Sf_N(z)|\,|dz|
  \leq\liminf_{N\to\infty} C(Sf_N)(x),\quad x\in Q, \\
  C(v) & \leq\liminf_{N\to\infty} C(Sf_N), \qquad
  \|C(v)\|_p  \leq\liminf_{N\to\infty} \|C(Sf_N)\|_p.
\end{split}
\end{equation*}
In combination with \eqref{eq:RHS}, this shows that \eqref{eq:SbdCWqp} implies the estimate
\begin{equation*}
  \|C(v)\|_p  \lesssim\|g\|_p.
\end{equation*}
If $g$ is real-valued, by a direct computation one checks that $|v|=|\nabla u|$, where $u$ is the Poisson extension of $g$, and the previous estimate can be rewritten as \eqref{eq:Cup<gp}. By splitting into real and imaginary parts, it is immediate that if this holds for real $g$, it also holds for complex $g$.

It remain to show that estimate \eqref{eq:Cup<gp} cannot hold. Using $CF(x)\geq 1_{Q}(x)\int_{Q^{\operatorname{ca}}} |F| dsdy$ with $Q=(0,1)$ and $|\nabla u|\geq|\partial_t u|$, we find that \eqref{eq:Cup<gp} would in particular imply that
\begin{equation}\label{eq:Cu1<gp}
  \int_0^1\|\partial_t u(t,\cdot)\|_{L_1(0,1)} dt\lesssim\|g\|_p.
\end{equation}
Let us consider random functions of the form $g=\sum_k \varepsilon_k g_k$, where the $\varepsilon_k$ are independent random signs with $P(\varepsilon_k=-1)=P(\varepsilon_k=+1)=\frac12$. Denoting by $u_k$ the Poisson extension of $g_k$, \eqref{eq:Cu1<gp} would then imply, using basic properties of random sums (Kahane's contraction principle in the first step, and Khintchine's inequality in the last one; see \cite[Prop. 3.2.10]{HNVW1}
and \cite[Thm. 6.1.13, Prop. 6.3.3]{HNVW2}), that
\begin{equation}\label{eq:Cu1<gp2}
\begin{split}
  \int_0^1\max_k\|\partial_t u_k(t,\cdot)\|_{L_1(0,1)} dt
  &\leq\int_0^1 E\Big\|\sum_k \varepsilon_k\partial_t u_k(t,\cdot)\Big\|_{L_1(0,1)} dt  \\
   &\lesssim E\Big\|\sum_k \varepsilon_k g_k\Big\|_{p} 
  \lesssim \Big\|\Big(\sum_k |g_k|^2\Big)^{1/2}\Big\|_{p},
\end{split}
\end{equation}
where $E$, appearing only in the intermediate steps, denotes the mathematical expectation. (The advantage of \eqref{eq:Cu1<gp2} in view of reaching a contradiction is that we have eliminated interactions between different functions $g_k$; in particular, on the left, we only need to check that individual terms are large for appropriate parameter values, and we do not need to worry about this largeness being destroyed by the largeness of some other term of opposite sign.)

We will now construct appropriate $g_k$ to see the failure of this derived estimate \eqref{eq:Cu1<gp2}. Let
\begin{equation*}
  g_k(x)=\phi(x)e^{i2\pi 2^k x},
\end{equation*}
where $\phi$ is a test function to be specified shortly. Note that $|g_k|=|\phi|$. Using the normalisation $\hat f(\xi)=\int_{\R}f(x)e^{-i2\pi x\xi}dx$ for the Fourier transform,  the Poisson integral $P_t f(x)$ has Fourier transform $\widehat{P_t f}(\xi)=e^{-2\pi t|\xi|}\hat f(\xi)$, and its time derivative the transform $\widehat{\dot P_t f}(\xi)=-2\pi|\xi|e^{-2\pi t|\xi|}\hat f(\xi)$. Hence
\begin{equation*}
  \hat g_k(\xi)=\hat\phi(\xi-2^k),\qquad
  \widehat{\dot P_t g_k}(\xi)=-2\pi |\xi|e^{-2\pi t|\xi|}\hat\phi(\xi-2^k).
\end{equation*}
Let us now choose a Schwartz test function $\phi$, not identically zero, such that the support of $\hat\phi$ is a compact subset of $\R_+$. Thus $\phi(\xi-2^k)$ is only non-zero for $\xi>2^k>0$, and hence we can drop the absolute values on $\xi$ above. We have
\begin{equation*}
\begin{split}
  \widehat{\dot P_t g_k}(\eta+2^k)
  &=-2\pi (\eta+2^k)e^{-2\pi t(\eta+2^k)}\hat\phi(\eta) \\
  &=2^{-2\pi t 2^k}(-2\pi\eta e^{-2\pi t\eta}\hat\phi(\eta)-2\pi 2^k e^{-2\pi t\eta}\hat\phi(\eta)) \\
  &=2^{-2\pi t 2^k}(\widehat{\dot P_t\phi}(\eta)-2\pi 2^k\widehat{P_t\phi}(\eta)).
\end{split}
\end{equation*}
Thus, inverting the Fourier transform,
\begin{equation*}
  \dot P_t g_k(x)e^{-i2\pi 2^k x}
  =2^{-2\pi t 2^k}(\dot P_t\phi(x)-2\pi 2^k P_t\phi (x)),
\end{equation*}
and hence
\begin{equation*}
\begin{split}
  |\dot P_t g_k(x)|
  &\geq 2^{-2\pi t 2^k}(2\pi 2^k |P_t\phi (x)|-|\dot P_t\phi(x)|), \\
  \|\dot P_t g_k\|_{L_1(0,1)}
  &\geq 2^{-2\pi t 2^k}(2\pi 2^k \|P_t\phi \|_{L_1(0,1)}-\|\dot P_t\phi\|_{L_1(0,1)}).
\end{split}
\end{equation*}

We claim that $\|P_t\phi \|_{L_1(0,1)}\geq c_0$ and $\|\dot P_t\phi \|_{L_1(0,1)}\leq c_1$ for some positive constants depending only on the choice of $\phi$, but not on $t\in(0,1)$. For the first one, note that $\widehat{P_t\phi}(\xi)=e^{-2\pi\xi t}\hat\phi(\xi)$ is not identically zero, and its support, the same as that of $\hat\phi$, is a compact subset of $\R_+$. Hence $P_t\phi(x)$ is not identically zero on $\R$, and it extends to an analytic function, hence it is also not identically zero for $x\in(0,1)$, and thus $\|P_t\phi\|_{L_1(0,1)}>0$ for every $t\in[0,1]$. Since $t\mapsto P_t\phi$ is continuous from $[0,1]$ to $L_1(0,1)$, the function $t\mapsto \|P_t\phi\|_{L_1(0,1)}$ attains a minimum on $[0,1]$, and this gives the number $c_0>0$. For the second claim, note that $\widehat{\dot P_t\phi}(\xi)=-2\pi \xi e^{-2\pi\xi t}\hat\phi(\xi)=ie^{-2\pi\xi t}(i2\pi\xi)\hat\phi(\xi)=i\widehat{P_t(\phi')}(\xi)$, and hence
\begin{equation*}
  \|\dot P_t\phi \|_{L_1(0,1)}\leq\|P_t(\phi') \|_{L_1(\R)}\leq \| \phi' \|_{L_1(\R)}=c_1.
\end{equation*}
Using these bounds, we conclude that
\begin{equation*}
   \|\partial_t u_k(t,\cdot)\|_{L_1(0,1)}= \|\dot P_t g_k\|_{L_1(0,1)}\gtrsim 2^{-2\pi t 2^k} 2^k ,\qquad k> k_0,\quad t\in(0,1).
\end{equation*}

With $k=k_0+1,\ldots,k_0+K$, we can now see the failure of \eqref{eq:Cu1<gp2}. For $t\in(2^{-j},2^{1-j})$ and $j=k_0+1,\ldots,k_0+K$, we find that
\begin{equation*}
  \max_k\|\partial_t u_k(t,\cdot)\|_{L_1(0,1)}
  \geq\|\partial_t u_j(t,\cdot)\|_{L_1(0,1)}\gtrsim 2^{-2\pi t 2^j} 2^j\gtrsim 2^j,
\end{equation*}
and hence
\begin{equation*}
  \int_0^1\max_k\|\partial_t u_k(t,\cdot)\|_{L_1(0,1)}dt
  \geq\sum_{j=k_0+1}^{k_0+K}\int_{2^{-j}}^{2^{1-j}}\max_k\|\partial_t u_k(t,\cdot)\|_{L_1(0,1)}dt\gtrsim \sum_{j=k_0+1}^{k_0+K}1=K,
\end{equation*}
while, on the other hand, we have
\begin{equation*}
  \Big\|\Big(\sum_{k=k_0+1}^{k_0+K} |g_k|^2\Big)^{1/2}\Big\|_{p}
  =\Big\|\Big(\sum_{k=k_0+1}^{k_0+K} |\phi|^2\Big)^{1/2}\Big\|_{p}
  =\Big\| \sqrt{K}|\phi|\Big\|_{p} \lesssim\sqrt{K}.
\end{equation*}
Thus \eqref{eq:Cu1<gp2} would imply that $K\lesssim\sqrt{K}$, which is clearly false. This contradiction shows the failure of each of the estimates \eqref{eq:Cup<gp}, \eqref{eq:SbdCWqp-prelim}, \eqref{eq:SbdCWqp}, \eqref{eq:Cu1<gp}, and \eqref{eq:Cu1<gp2}.

\bibliographystyle{acm}

\end{document}